\documentclass[a4paper,12pt]{article}

\usepackage{hyperref}

\usepackage{fullpage}
\usepackage{graphicx}
\usepackage{amsmath}
\usepackage{amssymb}
\usepackage{subfigure}
\usepackage{epsfig}

\usepackage{caption,color} 
\usepackage{amsthm}

\usepackage{bm}

\usepackage{url}

\usepackage{accents}

\theoremstyle{plain}
\newtheorem{thm}{Theorem}[section]
\newtheorem{lem}[thm]{Lemma}

\theoremstyle{remark}

\author{
Li-Chang Hung\footnote{Corresponding author's email address: \texttt{lichang.hung@gmail.com
} 
         }
\vspace{10mm}
\\
 \small
 \textit{
 Department of Mathematics, National Taiwan University, Republic of Taiwan}
\\
\\
\\
\\
 \small 
}

\title{Remark on the N-barrier method for a class of autonomous elliptic systems
}

\date{
\small 
Dedicated to my grandmother in memoriam
}

\begin{document}
\maketitle

\begin{abstract}


In this note, we aim to extend the previous work on an N-barrier maximum principle  (\cite{hung2015n,hung2015maximum}) to a more general class of systems of two equations. Moreover, an N-barrier maximum principle for systems of three equations is established.

\end{abstract}



\section{N-barrier maximum principle for two equations}\label{sec: NBMP 2 eqns}
\vspace{5mm}

We are concerned with the autonomous elliptic system in $\mathbb{R}$:
\begin{equation}\label{eqn: L-V BVP  before scaled without BCs}
\begin{cases}
d_1\,u_{xx}+\theta\,u_{x}+u\,f(u,v)=0, \quad x\in\mathbb{R}, \\ \\
d_2\,v_{xx}\hspace{0.8mm}+\theta\,v_{x}+v\,g(u,v)=0,\quad x\in\mathbb{R}, 
\end{cases}
\end{equation}
which arises from the study of traveling waves in the following reaction-diffusion system (\cite{Volperts94TWS-Parabolic}):
\begin{equation}\label{eqn: compet L-V sys of 2 species with diffu}
\begin{cases}
u_t=d_1\,u_{yy}+u\,f(u,v), \quad y\in\mathbb{R},\quad t>0,\\ \\
\hspace{0.7mm}v_t=d_2\,v_{yy}+v\,g(u,v), \quad y\in\mathbb{R},\quad t>0.\\
\end{cases}
\end{equation}
A positive solution $(u(x),v(x))=(u(y,t),v(y,t))$, $x=y-\theta \,t$ of \eqref{eqn: L-V BVP  before scaled without BCs} stands for a traveling wave solution of \eqref{eqn: compet L-V sys of 2 species with diffu}. Here $d_1$ and $d_2$ represent the diffusion rates and $\theta$ is the propagation speed of the traveling wave. Throughout, we assume, unless otherwise stated, that the following hypotheses on $f(u,v)\in C^{0}(\mathbb{R^{+}}\times\mathbb{R^{+}})$ and $g(u,v)\in C^{0}(\mathbb{R^{+}}\times\mathbb{R^{+}})$ are satisfied:
\begin{itemize}
  \item [$\mathbf{[H1]}$] the unique solution of $f(u,v)=g(u,v)=0$ and $u,v>0$ is  $(u,v)=(u^\ast,v^\ast)$;
  \item [$\mathbf{[H2]}$] the unique solution of $f(u,0)=0$ and $u>0$ is $u=u_1>0$; the unique solution of $f(0,v)=0, v>0$ is $v=v_1>0$; the unique solution of $g(u,0)=0, u>0$ is $u=u_2>0$; the unique solution of $g(0,v)=0, v>0$ is $v=v_2>0$; 
  \item [$\mathbf{[H3]}$]  as $u,v>0$ are sufficiently large, $f(u,v), g(u,v)<0$; as $u,v>0$ are sufficiently small, $f(u,v), g(u,v)>0$;
  \item [$\mathbf{[H4]}$] the two curves 
  \begin{equation}
  \mathcal{C}_f=\Big\{ (u,v)\;\big|\; f(u,v)=0,\;u,v\geq0\Big\}
  \end{equation}
  and
  \begin{equation}
  \mathcal{C}_g=\Big\{ (u,v)\;\big|\; g(u,v)=0,\;u,v\geq0\Big\}
  \end{equation}   
lie completely within the region $\mathcal{R}$, which is enclosed by the $u$-axis, the $v$-axis, the line $\bar{\mathcal{L}}$ given by
  \begin{equation}
  \bar{\mathcal{L}}=\big\{ (u,v)\;\big|\; \frac{u}{\bar{u}}+\frac{v}{\bar{v}}=1,\; u,v\geq0 \big\}
  \end{equation}
  and the line $\underaccent\bar{\mathcal{L}}$ given by
    \begin{equation}
  \underaccent\bar{\mathcal{L}}=\big\{ (u,v)\;\big|\; \frac{u}{\underaccent\bar{u}}+\frac{v}{\underaccent\bar{v}}=1,\;u,v\geq0\big\},
  \end{equation}
  for some $\bar{u}>\underaccent\bar{u}>0$, $\bar{v}>\underaccent\bar{v}>0$. 

\end{itemize}

\vspace{5mm}

In this note, the following boundary value problem for \eqref{eqn: L-V BVP  before scaled} is studied:
\begin{equation}\label{eqn: L-V BVP  before scaled}
\begin{cases}
\vspace{3mm}
d_1\,u_{xx}+\theta\,u_{x}+u\,f(u,v)=0, \quad x\in\mathbb{R}, \\
\vspace{3mm}
d_2\,v_{xx}\hspace{0.8mm}+\theta\,v_{x}+v\,g(u,v)=0,\quad x\in\mathbb{R}, \\
(u,v)(-\infty)=\text{\bf e}_{-},\quad (u,v)(+\infty)= \text{\bf e}_{+},
\end{cases}
\end{equation}
where $\text{\bf e}_{-}, \text{\bf e}_{+} =(0,0), (u_1,0), (0,v_2)$, or $(u^\ast,v^\ast)$. We call a solution $(u(x),v(x))$ of \eqref{eqn: L-V BVP  before scaled} an $(\text{\bf e}_{-},\text{\bf e}_{+})$-wave. As in \cite{hung2015n,hung2015maximum}, adding the two equations in \eqref{eqn: L-V BVP  before scaled} leads to an equation involving $p(x)=\alpha\,u(x)+\beta\,v(x)$ and $q(x)=d_1\,\alpha\,u(x)+d_2\,\beta\,v(x)$, i.e. 
\begin{align}\label{eq: q''+p'+f+g=0 before scale}
  0&=\alpha\,\big(d_1\,u_{xx}+\theta\,u_{x}+u\,f(u,v)\big)+
      \beta\,\big(d_2\,v_{xx}+\theta\,v_{x}+v\,g(u,v)\big)\notag\\[3mm]
  &=q''(x)+\theta\,p'(x)+ 
  \alpha\,u\,f(u,v)+
    \beta\,v\,g(u,v)\notag\\[3mm]
  &=q''(x)+\theta\,p'(x)+F(u,v),
\end{align}
where $\alpha$, $\beta>0$ are \textit{arbitrary} constants and $F(u,v)=\alpha\,u\,(\sigma_1-c_{11}\,u-c_{12}\,v)+\beta\,v\,(\sigma_2-c_{21}\,u-c_{22}\,v)$. To show the main result, we begin with a useful lemma.

\vspace{5mm}

\begin{lem}\label{lem: F(u,v)=0 is contained in R}
Under $\mathbf{[H1]}$$\sim$$\mathbf{[H4]}$, the curve $F(u,v)=0$ in the first quadrant of the $uv$-plane, i.e
\begin{equation}
\mathcal{C}_F=\Big\{ (u,v)\;\big|\; F(u,v)=0,\;u,v\geq0\Big\}
\end{equation}
lies completely within the region $\mathcal{R}$.  
\end{lem}

\begin{proof}
The proof is elementary and is thus omitted here.
\end{proof}

\vspace{5mm}

We are now in the position to prove 

\vspace{5mm}

\begin{thm}[\textbf{N-Barrier Maximum Principle for Systems of Two Equations}]\label{thm: N-Barrier Maximum Principle for 2 Species}
Assume $\mathbf{[H1]}$$\sim$$\mathbf{[H4]}$ hold. Suppose that there exists a pair of positive solutions $(u(x),v(x))$ of the boundary value problem
\begin{equation}\label{eqn: L-V BVP  before scaled}
\begin{cases}
\vspace{3mm}
d_1\,u_{xx}+\theta\,u_{x}+u\,f(u,v)=0, \quad x\in\mathbb{R}, \\
\vspace{3mm}
d_2\,v_{xx}\hspace{0.8mm}+\theta\,v_{x}+v\,g(u,v)=0,\quad x\in\mathbb{R}, \\
(u,v)(-\infty)=\text{\bf e}_{-},\quad (u,v)(+\infty)= \text{\bf e}_{+},
\end{cases}
\end{equation}
where $\text{\bf e}_{-}, \text{\bf e}_{+} =(0,0), (u_1,0), (0,v_2)$, or $(u^\ast,v^\ast)$. For any $\alpha,\beta>0$, let $p(x)=\alpha\,u(x)+\beta\,v(x)$ and $q(x)=\alpha\,d_1\,u(x)+\beta\,d_2\,v(x)$. We have
\begin{equation}\label{eqn: upper and lower bounds of p}
\underaccent\bar{p}
\leq p(x)\leq 
\bar{p}, \quad x\in\mathbb{R},
\end{equation}
where 
\begin{equation}
\bar{p}=\max\big(\alpha\,\bar{u},\beta\,\bar{v}\big)\,\frac{\max(d_1,d_2)}{\min(d_1,d_2)}
\end{equation}
and
\begin{equation}
\underaccent\bar{p}=\min\big(\alpha\,\underaccent\bar{u},\beta\,\underaccent\bar{v}\big)\,\frac{\displaystyle\min(d_1,d_2)}{\displaystyle\max(d_1,d_2)}\,\chi
,
\end{equation}
where
\begin{equation}
\chi
=
\begin{cases}
\vspace{3mm}
1,
\quad \text{if} \quad  \text{\bf e}_{\pm}\neq(0,0),\\
0,
\quad \text{if} \quad  \text{\bf e}_{\pm}=(0,0).
\end{cases}
\end{equation}




\end{thm}

\begin{proof}

We employ \textit{the N-barrier method} developed in \cite{hung2015n,hung2015maximum} to show \eqref{eqn: upper and lower bounds of p}. It is readily seen that once appropriate \textit{N-barriers} are constructed, the upper and lower bounds in \eqref{eqn: upper and lower bounds of p} are given in exactly the same way as in \cite{hung2015n,hung2015maximum}. 

To prove the lower bound, we can construct an N-barrier which starts either from the point $(\underaccent\bar{u},0)$ or $(0,\underaccent\bar{v})$. Due to Lemma~\ref{lem: F(u,v)=0 is contained in R}, the N-barrier is away from the region $\mathcal{R}$. Then the lower bound can be proved by contradiction as in \cite{hung2015n,hung2015maximum}. 

On the other hand, the upper bound can be shown in the same manner by constructing an N-barrier starting either from the point $(\bar{u},0)$ or $(0,\bar{v})$. In addition, it is clear that when the boundary conditions $\text{\bf e}_{+}=(0,0)$ or $\text{\bf e}_{-}=(0,0)$, a trivial lower bound of $p(x)$, i.e. $p(x)\geq0$ can only be given. This completes the proof.



\end{proof}

\vspace{5mm}
\setcounter{equation}{0}
\setcounter{figure}{0}
\section{N-barrier maximum principle for three equations}\label{sec: NBMP 3 eqns}
\vspace{5mm}

In this section, the following assumptions on $f(u,v,w)\in C^{0}(\mathbb{R^{+}}\times\mathbb{R^{+}}\times\mathbb{R^{+}})$, $g(u,v,w)\in C^{0}(\mathbb{R^{+}}\times\mathbb{R^{+}}\times\mathbb{R^{+}})$, and $h(u,v,w)\in C^{0}(\mathbb{R^{+}}\times\mathbb{R^{+}}\times\mathbb{R^{+}})$ are satisfied:
\begin{itemize}
  \item [$\mathbf{[A1]}$] \textbf{\emph{(Unique coexistence state)}} 
  $(u,v,w)=(u^\ast,v^\ast,w^\ast)$ is the unique solution of
  \begin{equation}
\begin{cases}
\vspace{3mm}
f(u,v,w)=0,\quad u,v,w>0, \\
\vspace{3mm}
g(u,v,w)=0,\quad u,v,w>0, \\
h(u,v,w)=0,\quad u,v,w>0. \\
\end{cases}
\end{equation}

  
  \item [$\mathbf{[A2]}$] \textbf{\emph{(Competitively exclusive states)}}
For some $u_i,v_i,w_i>0$ $(i=1,2,3.)$ with $(\Lambda_1-\Lambda_2)^2+(\Lambda_1-\Lambda_3)^2+(\Lambda_2-\Lambda_3)^2\neq0$ $(\Lambda=u,v,w.)$,
\begin{equation}
f(u_1,0,0)=f(0,v_1,0)=f(0,0,w_1)=0,
\end{equation}
\begin{equation}
g(u_2,0,0)=g(0,v_2,0)=g(0,0,w_2)=0,
\end{equation}
\begin{equation}
h(u_3,0,0)=h(0,v_3,0)=h(0,0,w_3)=0.
\end{equation}

   
  \item [$\mathbf{[A3]}$] \textbf{\emph{(Logistic-growth nonlinearity)}}
  For $u,v,w>0$, as $u,v,w$ are sufficiently large
  \begin{equation}
  f(u,v,w),g(u,v,w),h(u,v,w)<0,
  \end{equation}
  and  
  \begin{equation}
  f(u,v,w),g(u,v,w),h(u,v,w)>0,
  \end{equation}
  when $u,v,w$ are sufficiently small.
  \item [$\mathbf{[A4]}$] \textbf{\emph{(Covering pentahedron)}}
the three surfaces 
  \begin{equation}
  \mathcal{S}_f=\big\{ (u,v,w)\;\big|\; f(u,v,w)=0,\;u,v,w\geq0\big\},
  \end{equation}
  \begin{equation}
  \mathcal{S}_g=\big\{ (u,v,w)\;\big|\; g(u,v,w)=0,\;u,v,w\geq0\big\},
  \end{equation}  
  \begin{equation}
  \mathcal{S}_h=\big\{ (u,v,w)\;\big|\; h(u,v,w)=0,\;u,v,w\geq0\big\},
  \end{equation}
lie completely within the pentahedron (a polyhedron with five faces) $\mathcal{PH}$, which is enclosed by the $uv$-plane, the $uw$-plane, and the $vw$-plane as well as the planes 
  \begin{equation}
  \bar{\mathcal{P}}=\big\{ (u,v,w)\;\big|\; \frac{u}{\bar{u}}+\frac{v}{\bar{v}}+\frac{w}{\bar{w}}=1,\; u,v,w\geq0 \big\}
  \end{equation}
  and
    \begin{equation}
  \underaccent\bar{\mathcal{P}}=\big\{ (u,v,w)\;\big|\; \frac{u}{\underaccent\bar{u}}+\frac{v}{\underaccent\bar{v}}+\frac{w}{\underaccent\bar{w}}=1,\;u,v,w\geq0\big\}
  \end{equation}
for some $\bar{u}>\underaccent\bar{u}>0$, $\bar{v}>\underaccent\bar{v}>0$, $\bar{w}>\underaccent\bar{w}>0$. 
  


\end{itemize}

\vspace{5mm}

We can prove in a similar manner that an N-barrier maximum principle remains true for systems of three equations.

\vspace{5mm}

\begin{thm}[\textbf{N-Barrier Maximum Principle for Systems of Three Equations}]\label{thm: N-barrier Maximum Principle for 3 species}
Assume $\mathbf{[A1]}$$\sim$$\mathbf{[A4]}$ are satisfied.
Suppose that there exists a pair of positive solutions $(u(x),v(x),w(x))$ of the boundary value problem
\begin{equation}\label{eqn: L-V BVP  before scaled 3 eqns}
\begin{cases}
\vspace{3mm}
d_1\,u_{xx}+\theta\,u_{x}+u\,f(u,v,w)=0, \quad x\in\mathbb{R}, \\
\vspace{3mm}
d_2\,v_{xx}\hspace{0.8mm}+\theta\,v_{x}+v\,g(u,v,w)=0,\quad x\in\mathbb{R}, \\
\vspace{3mm}
d_3\,w_{xx}\hspace{0.0mm}+\theta\,w_{x}+w\,h(u,v,w)=0,\quad x\in\mathbb{R}, \\
(u,v,w)(-\infty)=\text{\bf e}_{-},\quad (u,v,w)(+\infty)= \text{\bf e}_{+},
\end{cases}
\end{equation}
where $\text{\bf e}_{-}, \text{\bf e}_{+} =(0,0,0), (u_1,0,0), (0,v_2,0), (0,0,w_3), (u^\ast,v^\ast,w^\ast)$. 
For any $\alpha,\beta,\gamma>0$, let $p(x)=\alpha\,u(x)+\beta\,v(x)+\gamma\,w(x)$. We have
\begin{equation}
\min\big(\alpha\,\underaccent\bar{u},\beta\,\underaccent\bar{v},\gamma\,\underaccent\bar{w}\big)
\frac{\min(d_1,d_2,d_3)}{\max(d_1,d_2,d_3)}\,\chi
\leq p(x)\leq 
\max\big(\alpha\,\bar{u},\beta\,\bar{v},\gamma\,\bar{w}\big)
\frac{\max(d_1,d_2,d_3)}{\min(d_1,d_2,d_3)}
\end{equation}
for $x\in\mathbb{R}$, where
\begin{equation}
\chi
=
\begin{cases}
\vspace{3mm}
1,
\quad \text{if} \quad  \text{\bf e}_{\pm}\neq(0,0,0),\\
0,
\quad \text{if} \quad  \text{\bf e}_{\pm}=(0,0,0).
\end{cases}
\end{equation}

\end{thm}

\vspace{5mm}





\vspace{5mm}










\end{document}